\documentclass[12pt]{amsart}
\usepackage{amssymb,amsmath,exscale,latexsym,amsthm,graphics,overpic,hyphenat,mathtools,caption,subcaption}
\usepackage{color}
\usepackage{enumitem}
\usepackage{epsfig}    
\usepackage{epic}      
\usepackage{eepic} 
\usepackage[matrix,arrow,curve,cmtip]{xy} 
\usepackage{letltxmacro}
\usepackage{ifthen} 
 
\usepackage{xcolor}
\usepackage{hyperref}
   
\makeatletter
\renewcommand{\l@figure}{\@tocline{0}{3pt plus2pt}{0pt}{3em}{}}
\makeatother

\makeatletter
\newtheorem*{rep@theorem}{\rep@title}
\newcommand{\newreptheorem}[2]{%
\newenvironment{rep#1}[1]{%
 \def\rep@title{#2 \ref{##1}}%
 \begin{rep@theorem}}%
 {\end{rep@theorem}}}
\newtheorem*{rep@prop}{\rep@title}
\newcommand{\newrepprop}[2]{%
\newenvironment{rep#1}[1]{%
 \def\rep@title{#2 \ref{##1}}%
 \begin{rep@prop}}%
 {\end{rep@prop}}}
\newtheorem*{rep@cor}{\rep@title}
\newcommand{\newrepcor}[2]{%
\newenvironment{rep#1}[1]{%
 \def\rep@title{#2 \ref{##1}}%
 \begin{rep@cor}}%
 {\end{rep@cor}}}

\makeatother

\setenumerate{itemsep=3pt,topsep=3pt}
\setenumerate[1]{label=\upshape(\roman*)}

\theoremstyle{plain} 
        \newtheorem{theorem}{Theorem}[section]
        \newtheorem*{theorem*}{Theorem}
        \newreptheorem{theorem}{Theorem}
        \newtheorem*{conj*}{Conjecture} 
        \newtheorem{lemma}[theorem]{Lemma}
        \newtheorem{prop}[theorem]{Proposition}
         \newtheorem{problem}[theorem]{Problem}
        \newrepprop{prop}{Proposition}
        \newtheorem{cor}[theorem]{Corollary}
        \newrepcor{cor}{Corollary}

        \newtheorem*{quasi_u_problem*}{The Quasisymmetric Uniformization Problem}

\theoremstyle{definition}
        \newtheorem{definition}[theorem]{Definition}

\theoremstyle{remark}

\numberwithin{equation}{section}
\numberwithin{figure}{section}

\newcommand{\diam}  {\operatorname{diam}}

\newcommand{\id} {\operatorname{id}}

\newcommand{\lcm} {\operatorname{lcm}}

\newcommand{\R}{\mathbb{R}}  
\newcommand{\C}{\mathbb{C}}      
  
\newcommand{\N}{\mathbb{N}}      
\newcommand{\Z}{\mathbb{Z}}      
\newcommand{\T}{\mathbb{T}}      
  
\newcommand{\CDach}{\widehat{\mathbb{C}}}
\newcommand{\D}{\mathbb{D}}      

\providecommand{\abs}[1]{\lvert#1\rvert}

\renewcommand{\:}{\colon}
\newcommand{\ra}{\rightarrow}
\newcommand{\sub}{\subset}
\newcommand{\eps}{\epsilon}

\newcommand{\ga}{\gamma}

\newcommand{\iu}{\textbf{\textit{i}}}
\newcommand{\crit}{\operatorname{crit}}
\newcommand{\post}{\operatorname{post}}

\newcommand{\mesh}{\operatorname{mesh}}

\newcommand{\CC}{\mathcal{C}}

\newcommand{\OC}{\mathcal{O}}

\newcommand{\Gtr}{G_{\text{tr}}}

\title{Quotients of Torus Endomorphisms and Latt\`es-Type Maps}%

\author{Mario Bonk}\thanks{Mario Bonk was partially supported by NSF grant DMS-1808856.}
\address{Department of Mathematics, University of California, 
Los Angeles, CA 90095, USA}
\email{mbonk@math.ucla.edu}

\author{Daniel Meyer}
\address{Department of Mathematical Sciences, University of Liverpool,
  Liverpool L69 7ZL,  UK}
\email{dmeyermail@gmail.com}

\date{\today}

\begin{document}
\abstract{We show that if an expanding  Thurston map is  the quotient of a torus endomorphism, then it has a parabolic orbifold and is a Latt\`es-type map.}
  
\endabstract

\maketitle

\tableofcontents

\section{Introduction}
\label{sec:introduction}

The main purpose of this paper is to present an open problem
about Thurston maps that has mystified the authors while writing \cite{BM17}. The
general underlying question  is which  properties of a Thurston map are of a purely topological  nature, or whether 
more geometric or even analytic structure is required to
characterize a property.
Our problem is closely related to certain classes of maps,
namely Latt\`es and Latt\`es-type maps. We start with recalling some background about these classes.

Latt\`es maps are rational maps  on the Riemann
sphere $\CDach=\C\cup\{\infty\}$ that are given as  quotient maps of holomorphic torus endomorphisms.  More precisely, a map
$f\: \CDach \ra \CDach$  is a Latt\`es maps if and only if there exist 
  a (non-homeomorphic and non-constant) holomorphic map 
$\overline{A}\colon \T\to \T$ on a complex torus $\T$   and a non-constant   holomorphic map
$\overline \Theta\: \T\ra \CDach$, such that   we have the following
commutative  diagram: 
\begin{equation}
  \label{eq:def_lattes3}
  \xymatrix{
    \T\ar[r]^{\overline{A}} \ar[d]_{\overline{\Theta}} 
    & \T \ar[d]^{\overline{\Theta}}
    \\
    \CDach \ar[r]^{f} & \CDach\rlap{.}
  }
\end{equation}
Here, a {\em complex torus}  $\T$ is a 
Riemann surface whose underlying $2$-manifold is a
$2$-dimensional torus.

It is then not hard to see that $f$ is a holomorphic map and hence a rational map on $\CDach$.
Moreover, one can show 
(see Theorem~\ref{thm:Lattesstruc}) that every Latt\`es map  $f$ is actually a postcritically-finite
rational map with
a parabolic orbifold (we explain this terminology in Section~\ref{sec:background}).  Verifying  that a map $f$ as
in \eqref{eq:def_lattes3} has indeed  a parabolic orbifold is  the difficult part in the proof of this statement.  
The
argument uses the holomorphicity of $f$ in an
essential way (see  \cite{Mi06} and \cite[p.~64]{BM17}).

Thurston raised the question when a map that behaves as a
rational map in a certain topological way is actually
``equivalent'' to a rational map (see \cite{DH} and \cite{BM17}
for a systematic study of this point of view).  In view of this,
it is natural to consider topological analogs of maps as in
\eqref{eq:def_lattes3}. This means, we consider maps
$f\colon S^2\to S^2$ with the property that there exists a torus
endomorphism $\overline{A} \colon T^2\to T^2$ (i.e., an
unbranched covering map) with topological degree
$\deg(\overline A)\ge 2$, as well as a branched covering map
$\overline{\Theta}\colon T^2 \to S^2$ such that we have the
following commutative diagram:
\begin{equation}
  \label{eq:Lattesdef}
  \xymatrix{
    T^2 \ar[r]^{\overline A} \ar[d]_{\overline\Theta} &
    T^2 \ar[d]^{\overline\Theta}
    \\
    S^2 \ar[r]^f & S^2\rlap{.} 
  }
\end{equation}
Here, $S^2$ is a topological $2$-sphere, and $T^2$ is a topological $2$-torus. We use
notation different from  \eqref{eq:def_lattes3} to indicate
that these are topological objects and not Riemann surfaces,
meaning that the surfaces  are not equipped with a conformal structure. 

If a map 
$f$ arises as in
 \eqref{eq:Lattesdef}, then we call $f$ a \emph{quotient of
  a torus endomorphism} (see Definition~\ref{def:quotient_torus} for a precise statement). One can show that such a map $f$ is 
  actually a Thurston map, i.e., a non-homeomorphic branched covering map with a finite set of postcritical points (see Lemma~\ref{lem:simobserv}). One should expect that these maps are closely related to Latt\`es maps. In particular, one expects a positive answer to the following question. 
  
\begin{problem}
  \label{que:torus_endo_para}
  Does every quotient of a torus endomorphism have a parabolic
  orbifold? 
\end{problem}
We  have repeatedly tried
 to tackle this  problem and also consulted with various experts, but 
a convincing argument for a positive answer is elusive at this point. Accordingly, it seems appropriate to present  a partial answer and some facts related to Problem~\ref{que:torus_endo_para}. This is the main purpose of this paper.

To formulate our result, we first have to define  Latt\`es-type
maps. As this involves somewhat technical terminology,  we will introduce these maps in an informal way for now, but 
will give a precise definition later in Section~\ref{sec:background} (see Definition~\ref{def:Lattestype}).

As a starting point, one  notices (see the discussion at the beginning of Section~\ref{sec:para}) that by a lifting argument 
for each map $f\: S^2\ra S^2$ as in \eqref{eq:Lattesdef}, one has a commutative diagram of the form: 
 \begin{equation}
   \label{eq:def_lattes_type}
   \xymatrix{
     \R^{2} \ar[r]^{A} \ar[d]_{\Theta} & \R^{2} \ar[d]^{\Theta}
     \\
     S^{2} \ar[r]^{f} & S^{2} \rlap{.} 
   }
 \end{equation}
 Here, $\Theta\: \R^2\ra S^2$ is a branched covering map and
 $A\:\R^2\ra \R^2$ is an orientation-preserving homeomorphism
 with a suitable equivariance property with respect to the group
 $G$ of deck transformations of $\Theta$.
 
 In the special situation of \eqref{eq:def_lattes_type} when  $A$ is a (real)  affine map on $\R^2$ and $G$ is a crystallographic group, one calls $f$ a 
 {\em Latt\`es-type map} (see  Definition~\ref{def:Lattestype}). 
 One can show that each Latt\`es map  is also a  Latt\`es-type map.
 This immediately follows from the characterization of Latt\`es maps as in   condition \ref{item:Lattessruciii}  of Theorem~\ref{thm:Lattesstruc}.
  
 Moreover,  each   Latt\`es-type map is a quotient of a torus endomorphism with a parabolic orbifold  (see Proposition~\ref{prop:immediate}). In general, the converse implication is  not true, but our main result provides such a  converse under the assumption that the Thurston map $f$ is {\em expanding} (see  \eqref{eq:mesh} for the precise definition).

\begin{theorem}
  \label{thm:quo}
  Let $f\: S^2\ra S^2$ be an expanding Thurston map. Then the
  following conditions are equivalent:
  
  \begin{enumerate}
 
 \item \label{item:quo1} $f$ is the quotient of a torus endomorphism. 
    
 \item \label{item:quo2}    $f$ has a parabolic orbifold. 
 \item 
   \label{item:quo3}
   $f$ is a Latt\`es-type map. 
 \end{enumerate} 
\end{theorem}
Since every quotient $f$ of a torus endomorphism is actually a Thurston map, the previous statement gives an answer to Problem~\ref{que:torus_endo_para} if $f$ is expanding.  

As we already pointed out, the implication   
\ref{item:quo3}$\Rightarrow$\ref{item:quo1} is known (by Proposition~\ref{prop:immediate}). The most  difficult part in the proof  of Theorem~\ref{thm:quo} is to establish the implication  
\ref{item:quo1}$\Rightarrow$\ref{item:quo2}. Here we cannot rely on holomorphicity as in the proof of the parabolicity of the orbifold of Latt\`es maps as defined in \eqref{eq:def_lattes3}. Instead, we will use  a dynamical argument based  on the expansion properties of  $f$ and its associated maps (see the considerations in Section~\ref{sec:para} which lead to Proposition~\ref{prop:quottorendo=para}).

The proof of the implication
\ref{item:quo2}$\Rightarrow$\ref{item:quo3} relies on the fact
that an expanding Thurston map $f$ with parabolic orbifold cannot
have periodic critical points, which in turn implies that $f$ is
Thurston equivalent to a Latt\`{e}s-type map $g$ (see  Proposition~\ref{prop:fpara_Latt_type}). We will show
that $g$ is expanding (see the proof of Proposition~\ref{prop:para_Lattes}). A standard result 
(Theorem~\ref{thm:Th_equiv_top_conj}) then implies that $f$
and $g$ are in fact topologically conjugate. We can then  conclude that
$f$ itself is a Latt\`{e}s-type map (see
Lemma~\ref{lem:Lattes_type_conj}). 

One may ask to what extent some of these implications are true without the assumption that the Thurston map $f$ is expanding. For \ref{item:quo1}$\Rightarrow$\ref{item:quo2}, this leads to the open 
 Problem~\ref{que:torus_endo_para}. The implication  
\ref{item:quo3}$\Rightarrow$\ref{item:quo1} is still true without expansion (see Proposition~\ref{prop:immediate}). 
The relation between  \ref{item:quo2} and \ref{item:quo3} is covered by the following statement: {\em  Let $f$ be a Thurston map. Then 
 $f$ is Thurston equivalent to a Latt\`es-type map if and only if 
 $f$ has a parabolic orbifold and no periodic critical points}
 (see Proposition~\ref{prop:fpara_Latt_type}). Note that Thurston maps 
with   parabolic orbifolds and  periodic critical points are  also easy to classify up to Thurston equivalence: essentially, these are power
maps $z\mapsto z^n$ and Chebyshev polynomials (see 
\cite[Chapter 7]{BM17}).

The paper is organized as follows. We review all the relevant background and preliminaries in Section~\ref{sec:background}.
The proof of the implications \ref{item:quo1}$\Rightarrow$\ref{item:quo2} and \ref{item:quo2}$\Rightarrow$\ref{item:quo3} 
in Theorem~\ref{thm:quo} are then given in  Sections~\ref{sec:para} and ~\ref{sec:charLtype}. We wrap up the proof of  Theorem~\ref{thm:quo} at the end of Section~\ref{sec:charLtype}.

\subsection*{Notation}
\label{sec:notation}
When an object $A$ is defined to be another object $B$, we write $A\coloneqq B$
for emphasis. 

We denote by $\N=\{1,2,\dots\}$ the set of natural numbers and  by 
$\N_0=\{0,1, 2, \dots \}$
the set of natural numbers
including $0$. 
The sets 
of integers, real numbers, and complex numbers are denoted by $\Z$, 
 $\R$, and $\C$, respectively.  We let $\CDach\coloneqq \C\cup\{\infty\}$ be
the Riemann sphere.
 We also consider
$\widehat{\N}\coloneqq \N\cup\{\infty\}$. 
 If $A\sub \widehat{\N}$, then $\lcm(A)\in  \widehat{\N}$
denotes the least common multiple of the numbers in $A$.


When we consider two objects $A$ and $B$, and there is a natural
identification between them that is clear from the context, we
write 
$A\cong B$.\index{$\cong$} 
For example, $\R^2\cong \C$ if we identify a point $(x,y)\in \R^2$ with $x+ y\iu \in \C$, where $\iu$ is the imaginary unit.

The cardinality of a set $X$ is denoted by $\#X$ and the identity map on $X$ by $\id_X$. 
If $x_n\in X$ for $n\in \N$ are points in $X$, we  denote the sequence of these points 
  by $\{x_n\}_{n\in\N}$, or just
by $\{x_n\}$ if the index set $\N$ is understood.

  If $f\: X \ra X$ is a map and $n\in \N$, then 
$$f^n\coloneqq \underbrace{f\circ \dots \circ f}_{\text{$n$ factors} } $$ is the $n$-th iterate of $f$. We set $f^0\coloneqq \id_X$ for convenience.

Let $f\colon X\to Y$ be a map between sets $X$ and $Y$. If
$U\sub  X$, then 
$f|U$
stands for the \emph{restriction} of $f$ to $U$. If $A\sub Y$, then 
$f^{-1}(A)\coloneqq \{x\in X : f(x)\in A\}$ is the preimage of
$A$ in $X$. Similarly,  $f^{-1}(y)\coloneqq
\{x\in X : f(x)=y\}$ is the preimage of a point $y\in Y$. 

If $f\: X\ra X$ is a map, then preimages of a set $A\subset X$ or a point $p\in X$ under the
$n$-th iterate $f^n$ are denoted by $f^{-n}(A)\coloneqq \{x\in X
: f^n(x) \in A\}$ and $f^{-n}(p)\coloneqq \{x\in X : f^n(x)=p\}$, 
respectively. 

Let $(X,d)$ be a metric space, and $M\sub X$.
 Then we denote by 
 $$\diam_d(M)\coloneqq \sup \{d(x,y): x,y\in M\}$$ 
 the diameter of $M$. We drop the
subscript $d$ here if the metric $d$ is clear
from the context.

%

\section{Background}
\label{sec:background}
In this section we state some relevant definitions and collect some facts for the convenience of the reader. More details on all of these topics can be found in \cite{BM17}. 
 
\subsection{Branched Covering Maps} 
\label{sec:branched covers} We closely follow the presentation in \cite[Section 2.1 and Section A.6] {BM17}. A {\em surface} is a connected and oriented topological $2$-manifold. A surface $X$ is a {\em topological disk} if it is homeomorphic to the unit disk $\D\coloneqq\{ z\in \C: |z|<1\}$.  

Let $X$ and $Y$ be surfaces, and $f\: X\ra Y$ be a continuous
map.
Then $f $ is a {\em branched covering map}, if for each
point $q\in Y$, there exists
a topological disk $V\subset Y$ with
$q\in V$ that is {\em evenly covered}\index{evenly covered} by
$f$ in the following sense: for some index set $I\ne \emptyset$
we can write $f^{-1}(V)$ as a disjoint union
 $$f^{-1}(V)=\bigcup_{i\in I} U_i$$ of open  sets  $U_i\sub X$, such that $U_i$ 
 contains precisely one point $p_i\in f^{-1}(q)$. Moreover, we require that for each $i\in I$, there 
 exists $d_i\in \N$, 
  and orientation-preserving homeomorphisms
 $\varphi_i\: U_i \ra \D$ and $\psi_i\: V\ra \D$ with $\varphi_i(p_i)=0$ and $\psi_i(q)=0$, such that 
 \begin{equation}\label{eq:z^dbrcov}
  (\psi_i\circ f \circ \varphi_i^{-1})(z)=z^{d_i}
  \end{equation}
 for all $z\in \D$.

 For
 given $f$, the number $d_i$ is uniquely determined by $p=p_i$ and
 called the 
 {\em local degree}\index{local degree|textbf}\index{deg@$\deg_f(p), \deg(f,z)$|textbf} 
 of $f$ at $p$, denoted by $\deg(f,p)$. Our definition allows different local degrees
 at points in the same fiber $f^{-1}(q)$.

 Every branched covering map $f\: X\ra Y$ is surjective, 
 {\em open}
(images of open sets are open), and 
{\em discrete}
 (the preimage set  of every point is {\em discrete in} $X$, i.e., it has no limit points in $X$). Every (locally orientation-preserving) covering map (see \cite[Section A.5]{BM17}) is also a branched covering map.

 A 
 \emph{critical point}
 of a branched covering map $f\colon X\to Y$
 is a point $p\in X$ with $\deg(f,p)\ge 2$. A 
 {\em critical value}\index{critical!value} 
 is a point $q\in Y$, such that the fiber $f^{-1}(q)$ contains a critical point of $f$. 
The set of critical points of $f$ is discrete in $X$ and  the set of critical values of $f$ is discrete in $Y$.
If 
$f\colon X\to Y$ is a branched 
covering map, then $f$ is an orientation-preserving local homeomorphism
near each point $p\in X$ that is not a critical point of $f$. 

If $X$ is a compact surface and $f\: X\ra X$ is a branched covering map, then we denote by $\deg(f)\in \N$ the {\em topological degree}  
of $f$ (see  \cite[Section 2.1 and A.4]{BM17} for the precise  definition and more discussion). 

The following statement is useful if one has to deal with compositions of branched covering maps (see \cite[Lemma A.16]{BM17}). 
\begin{lemma}[Compositions of branched covering maps]
  \label{lem:2_3_branched}
    Let $X$, $Y$, and $Z$ be surfaces, and 
   $f\colon X\to Z$, $g\colon Y\to Z$, and 
  $h\colon X\to Y$ be continuous maps, such that $f=g\circ h$. 
  \begin{enumerate}
  \item 
    \label{item:2_out3_1}
    If $g$ and $h$ are branched covering maps,  and $Y$ and $Z$ 
     are compact, then $f$ is also a branched covering map.
  \item 
    \label{item:2_out_3_2}
    If $f$ and $g$ are branched covering maps, then $h$ is a
    branched covering map. Similarly,  if $f$ and $h$ are branched covering maps, then $g$
    is a branched covering map.
  \end{enumerate}
  \end{lemma}
  
Let $X$, $Y$, and $Z$ be surfaces, and 
   $h\colon X\to Y$, $g\colon Y\to Z$ be branched covering maps.
If  $g\circ h\: X\ra Z$ is also a 
  branched covering map, then we have
   \begin{equation}\label{eq:locdegmult}
 \deg(g\circ h, x)=\deg(g, h(x))\cdot \deg(h,x)
\end{equation}
for all $x\in X$. We will use this multiplicativity  of
local degrees throughout,  usually without specific reference. For
the proof we refer to \cite[Lemma A.17]{BM17}. Note that there
slightly stronger assumptions were used, but the proof for
\eqref{eq:locdegmult}  remains
valid without change. 


\subsection{Thurston Maps}
\label{sec:thurston-maps}

 Throughout this paper, $S^2$ denotes a topological  $2$-sphere.
  We assume that $S^2$ is 
equipped with a fixed  orientation. To be able to use metric language,
we also assume that $S^2$ carries  a base metric that induces the
given  topology on $S^2$.

Let $f\colon S^2\to S^2$ be a branched covering map. 
We denote by
$$\crit(f)\coloneqq\{p\in S^2:\deg(f,p)\ge 2\}$$ its (finite) set of critical points and by
\begin{equation*}
  \post(f) = \bigcup_{n\geq 1} \{f^n(c) : c\in \crit(f)\}
\end{equation*}
its set of postcritical points.  One can show  that
$\post(f^n)=\post(f)$. If $\post(f)$ is a finite set and
$\deg(f)\ge 2$, we call $f$ a \emph{Thurston map}. We also say
that $f$ is \emph{postcritically-finite}. If, in addition, $f$ is
defined on $\CDach$ and holomorphic, then $f$ is a
postcritically-finite rational map and we call $f$ a \emph{rational}
Thurston map. A {\em periodic point} of $f$ is a point
$p\in S^2$ with $f^n(p)=p$ for some $n\in \N$.  For more details
see \cite[Section 2.1]{BM17}.


\subsection{Expansion}
\label{sec:expansion}

Let $f\: S^2\ra S^2$ be a Thurston map. We say that $f$ is \emph{expanding}, if for some Jordan curve $\CC\subset S^2$ with $\post(f)\subset\CC$, we have 
\begin{equation}\label{eq:mesh}
  \lim_{n\to \infty} \mesh(f,n,\CC) =0.
\end{equation}
Here $\mesh(f,n,\CC)$ denotes the supremum of the diameters of
components of $S^2{\setminus }f^{-n}(\CC)$. This  condition 
is independent of
the choice of the curve $\CC$ and the base metric on $S^2$ (see \cite[Section 6.1]{BM17}). 

If $f\colon \CDach \to \CDach$ is a rational Thurston map, then
it is expanding if and only if $f$ has no periodic critical
points. This is the case if and only if 
its Julia set is the whole sphere $\CDach$ (see 
\cite[Proposition 2.3]{BM17}). 

Every expanding Thurston map $f\: S^2\ra S^2$ has an associated {\em visual metric} on $S^2$ that induces the given topology.
The metric $\varrho$ has an associated {\em expansion factor} $\Lambda>1$. We refer to \cite[Chapter 8]{BM17}
for precise definitions. We will only need one fact about visual metrics. 

\begin{lemma}
  \label{lem:liftpathshrinks}
  Let $f\:S^2\ra S^2$ be an expanding Thurston map, and $\varrho$
  be a visual metric for $f$ with expansion factor
  $\Lambda>1$. Then there exist constants $\delta_\varrho>0 $ and
  $C>0$ such that for each path $\alpha \colon [0,1]\to S^2$ with
  $\diam_\varrho(\alpha)<\delta_\varrho$, each $n \in \N$, and
  each path $\widetilde \alpha \: [0,1]\ra S^2$ with
  $f^n\circ \widetilde \alpha=\alpha$, we have
  \begin{equation*}
    \diam_\varrho (\widetilde \alpha) \le  C \Lambda^{-n}.
  \end{equation*}
\end{lemma}
This follows from \cite[Lemma 8.9]{BM17} and the discussion after
this lemma. In other words, if we lift a path with sufficiently
small diameter under $f^n$, then the lifts shrink uniformly at an
exponential rate as $n\to \infty$.

\subsection{Thurston Equivalence}
\label{sec:thurston-equivalence}

For Thurston maps one often considers the following notion of
equivalence (see \cite[Section~2.4]{BM17} for more explanations). 

\begin{definition}
  \label{def:Th_equiv}
  Let $f\colon S^2\to S^2$ and $\widehat f \colon \widehat{S}^2 \to
  \widehat{S}^2$ be Thurston maps. Then they are called
  \emph{Thurston equivalent} if there exist homeomorphisms
  $h_0,h_1\colon S^2\to \widehat{S}^2$ that are isotopic relative to 
  $\post(f)$ and satisfy $h_0\circ f = \widehat f \circ h_1$. 
\end{definition}

Here $\widehat{S}^2$ is another topological $2$-sphere.

Two maps $f\colon S^2\to S^2$ and
$\widehat f\colon \widehat{S}^2 \to \widehat{S}^2$ are called
{\em topologically conjugate} if there exists a homeomorphism
$h\: S^2 \ra \widehat S^2$ such that
$h\circ f = \widehat f \circ h$.  It easily follows from the
definitions that if two Thurston maps are topologically
conjugate, then they are Thurston equivalent. The converse is
true if the maps are expanding.

\begin{theorem}
  \label{thm:Th_equiv_top_conj}
    Let $f\colon S^2\to S^2$ and $\widehat f \colon \widehat{S}^2 \to
  \widehat{S}^2$ be expanding Thurston maps that are Thurston
  equivalent. Then they are topologically conjugate. 
\end{theorem}

This is \cite[Theorem~11.1]{BM17}.

\subsection{The Orbifold Associated with a Thurston Map}
\label{sec:parabolic-orbifolds}

We follow  \cite[Section 2.5]{BM17}. Let $f\: S^2\ra S^2$ be a Thurston map. 
For a given $p\in S^2$, we set
\begin{equation*}
  \alpha_f(p)=\lcm\{ \deg(f^n,q)
  :
  q\in S^2,\, n\in \N, \text{ and } f^n(q)=p\}. 
\end{equation*}
Here, $\lcm (M)\in \widehat\N\coloneqq \N\cup\{\infty\}$ denotes
the least common multiple of a set $M\sub \N$.  

Note that $\alpha_f(p) = \infty$ is possible. This is true if and only if
$p$ is contained in a critical cycle of $f$, i.e., $p$ is a fixed
point and a critical point of $f^n$ for some $n\in \N$. It
follows that $\alpha_f$ is {\em finite} (i.e., it does not take
the value $\infty$) if and only if $f$ has no periodic critical
points.  Note that, in general, an expanding Thurston map may have periodic critical
points 
(see \cite[Example~12.21]{BM17}), but not if it is rational (see \cite[Proposition~2.3]{BM17}).
One can also show that $\alpha_f(p) =1$ if and only if
$p\in S^2{\setminus }\post(f)$ (see \cite[Proposition 2.9]{BM17}).

The function $\alpha_f\colon S^2 \ra 
\widehat{\N}$ is called the \emph{ramification
  function} of $f$ and  $\OC_f= (S^2, \alpha_f)$ (i.e., the underlying $2$-sphere equipped with this ramification function)  the
\emph{orbifold} associated with $f$. The \emph{Euler
  characteristic} of $\OC_f$ is defined as
\begin{equation*}
  \chi(\OC_f) = 2 - \sum_{p\in S^2}
  \left( 1- \frac{1}{\alpha_f(p)} \right).
\end{equation*}
For a Thurston map $f$, we always have $\chi(\OC_f) \leq 0$
(see  \cite[Proposition 2.12]{BM17}).
We say that $f$ has a {\em parabolic orbifold} if 
$\chi(\OC_f) = 0$ and a {\em hyperbolic orbifold} if 
$\chi(\OC_f) <0$.

\subsection{Parabolic Orbifolds}
\label{sec:parabolic-orbifolds-1}

To give a more  precise classification of Thurston maps  with 
 parabolic orbifold, we consider the postcritical points $p_1,\dots,p_k$, $k\in \N$, of a Thurston map $f$ labeled
so  that 
$$\alpha_f(p_1) \leq \alpha_f(p_2)\leq\dots \le \alpha_f(p_k).$$ The $k$-tuple 
 $ (\alpha_f(p_1), \dots, \alpha_f(p_k))$
is called the \emph{signature} of $\OC_f$.

The orbifold $\OC_f$ associated with a Thurston map $f$
is parabolic, i.e., $\chi(\OC_f)=0$,
if and only
if the signature of  $\OC_f$ is in  the following list: \begin{equation}
\label{eq:list}
  (\infty,\infty), (2,2,\infty), (2,4,4), (2,3,6), (3,3,3), (2,2,2,2)  
\end{equation}
(see  \cite[Proposition 2.14]{BM17}). 

It follows from the definition  of the ramification function $\alpha_f$ that
$\deg(f,p)\alpha_f(p)$ divides $\alpha_f(f(p))$ for all
$p\in S^2$ (see  \cite[Proposition 2.8~(ii)]{BM17}). One can show that  $\OC_f$ is parabolic if and only if
\begin{equation}\label{eq:para}
  \deg(f,p)\alpha_f(p) = \alpha_f(f(p))
\end{equation}
for all $p\in S^2$ (see \cite[Proposition 2.14]{BM17}). 

\subsection{Latt\`{e}s Maps}
\label{sec:lattes-maps} We follow the presentation in  \cite[Chapter 3]{BM17}.  The definition of a Latt\`{e}s map is based on the following fact
(this is essentially  well known; see  \cite{Mi06} and \cite[Theorem 3.1]{BM17}). 

\begin{theorem}[Characterization of Latt\`es maps]
  \label{thm:Lattesstruc}  Let $f\: \CDach \ra \CDach$ be
  a map. Then the following conditions are equivalent:  

\begin{enumerate}
 
 \item \label{item:Lattessrucii} $f$ is a rational Thurston map that
   has a parabolic orbifold and no  periodic critical points. 
    
 \item \label{item:Lattessruciii} There exist a crystallographic
   group $G$, 
   a $G$-equivariant  holomorphic  map $A\: \C\ra \C$ of the form 
   $A(z)=\alpha z+\beta$, where $\alpha, \beta\in \C$, $\abs{\alpha}>
   1$, and a holomorphic map $\Theta\: \C\ra \CDach$ induced by 
   $G$, such that $f\circ \Theta=\Theta\circ A$. 
   
 \item 
   \label{item:Lattesruciv}
   There exist a complex torus $\T$, a holomorphic torus endomorphism
   $\overline{A} \colon \T\to \T$ with $\deg (\overline{A}) >1$, and a  non-constant 
   holomorphic map $\overline{\Theta} \colon \T \to \CDach$ such that
   $f\circ \overline{\Theta} = \overline{\Theta} \circ \overline{A}$. 
 \end{enumerate} 
\end{theorem}
Here a \emph{crystallographic group} $G$ is a subgroup of the
group of orien\-ta\-tion-preserving isometries of $\C$ that acts properly discontinuously and cocompactly on $\C$. In particular, 
 each
element $g\in G$ is a map of the form $g\colon z\in \C \mapsto
\alpha z + \beta$, where $\alpha,\beta\in \C$ with
$\abs{\alpha}=1$. Note that this definition of a  crystallographic group $G$ is more restrictive than usual, because we require that all elements $g\in G$ preserve orientation on $\C$.  These groups are completely classified (see \cite[Theorem 3.7]{BM17}).

A continuous map $\Theta\: \R^2 \to S^2$
is {\em induced} by a group $G$ of  homeo\-morphisms on $\R^2$ if, for $u,v\in \R^2$, we have 
$\Theta(u)=\Theta(v)$ if and only if there exists $g\in G$ such that $v=g(u)$. The reason for this terminology is that under some additional assumptions (for example, when $\Theta$ is surjective and open), there exists a homeomorphism between 
$S^2$ and the quotient space $\R^2/G$ such that $\Theta$ corresponds to the quotient map $\R^2\ra \R^2/G$ (see 
\cite[Corollary A.23]{BM17} for a precise statement along these
lines). 
Such maps
$\Theta$ are sometimes called {\em strongly $G$-automorphic} in the literature.

Finally, if $G$ is a group of  homeomorphisms  on $\R^2$ and 
 $A\: \R^2\ra \R^2$ is a homeomorphism, then $A$ is called 
 $G$-{\em equivariant} if 
$$ A\circ g \circ A^{-1} \in G$$
for each $g\in G$.

A map $f\colon \CDach \to \CDach$ on the Riemann sphere $ \CDach$
is a {\em Latt\`{e}s map} if one, hence each, of the conditions
in Theorem~\ref{thm:Lattesstruc} is satisfied. Such a map is
always expanding (see \cite[Proposition~2.3]{BM17}). Note that
condition~\ref{item:Lattesruciv} in this theorem was how we
introduced Latt\`es maps in the introduction.

In the following,  for $h\in \R^2\cong\C$, we denote by
$\tau_h\: \R^2\ra\R^2$ the translation  defined as 
\begin{equation} \label{eq:trans} 
\tau_h(u)=u+h, \quad \text{for $u\in \R^2$}.
\end{equation} 

The subgroup of all translations in a crystallographic $G$ is
denoted by $\Gtr$. One can show that
for each crystallographic $G$, there exists
a rank-$2$ lattice $\Gamma\sub \R^2$, such that 
$\Gtr=\{\tau_\ga: \ga\in \Gamma\} $. In particular, $\Gtr$ also acts cocompactly and properly discontinuously on $\R^2$. Moreover, 
the quotient space $\T=\C/\Gtr\cong \R^2/\Gamma$ is a torus carrying a natural complex structure, and it  is hence a complex torus.

%
%
%
%

If $f$ is a Latt\`es map, then one can always find $A$,
$\overline A$, $G$, $\overline \Theta$
as in Theorem~\ref{thm:Lattesstruc}, such that
we have the following commutative diagram (see \cite[(3.10)]{BM17}): 
\begin{equation}
  \label{eq:holom_qu_diagram}
  \xymatrix{
    \C \ar[r]^{A} \ar[d]_{\pi} \ar@/_2pc/[dd]_{\Theta} 
    & \C \ar[d]^{\pi}\ar@/^2pc/[dd]^{\Theta}
    \\
    \T \ar[r]^{\overline{A}}\ar[d]_{\overline{\Theta}} 
    & \T\ar[d]^{\overline{\Theta}}
    \\
    \CDach \ar[r]^{f} & \CDach\rlap{.}
  }
\end{equation}

Here $\pi\colon \C \to \C/\Gtr= \T$ is the quotient map which is the universal covering map of
the torus $\T$. As we will see, a topological analog of \eqref{eq:holom_qu_diagram} will be the starting point for the proof of Theorem~\ref{thm:quo}.

\subsection{Quotients of Torus Endomorphisms}
\label{sec:parabolicity}

We first record a precise definition for a quotient of a torus endomorphism.  As before, we will denote by $T^2$  a  $2$-dimensional topological torus. 
We call a branched  covering map $\overline A\: T^2\ra T^2$ 
a {\em torus endomorphism}. It easily follows from the Riemann-Hurwitz formula that $\overline A$ actually cannot have critical points, and so must be a (locally orientation-preserving) covering map. We can now give a precise definition of the most important concept in his paper.   

\begin{definition}[Quotients of torus endomorphisms]
  \label{def:quotient_torus}
  Let $f\colon S^2\to S^2$
  be a map on a $2$-sphere $S^2$ that satisfies
  the following
  condition: there exists a torus endomorphism
  $\overline A\: T^2\ra T^2$ with $\deg(\overline A)\ge 2$ and a
  branched covering map $\overline \Theta\: T^2\ra S^2$, such that
  $f\circ \overline \Theta= \overline\Theta\circ \overline A$.
  Then $f$ is called a {\em quotient of a torus endomorphism}.
\end{definition}
In this case,  we have a commutative diagram as in \eqref{eq:Lattesdef}. 

The following statement summarizes some facts about these 
maps (see \cite[Lemmas 3.12 and 3.13]{BM17}).

\begin{lemma}[Properties of quotients of torus endomorphisms]
  \label{lem:simobserv} 
  Let $f\: S^2\ra S^2$ be a quotient
  of a torus endomorphism, and $\overline \Theta\: T^2\ra S^2$ and 
  $\overline{A}\colon T^2\to T^2$ with $\deg(\overline A)\ge 2$ be as in
 Definition~\ref{def:quotient_torus}. Then the following statements are true:

\begin{enumerate}

 \item \label{item:simobi}
 The map $f$ is a Thurston map without periodic critical points,
 and it satisfies $\deg(f)=\deg(\overline A)\ge 2$. 
 
 \item \label{item:simobii}
   The set $\post(f)$ is equal to the set of critical values of
   $\overline \Theta$, i.e., 
   \begin{equation*}
     \post(f) = \overline{\Theta}(\crit(\overline{\Theta})).     
   \end{equation*}

 \item 
   \label{item:simobiii}
     $f$ has a  parabolic orbifold if  
and only if   \begin{equation*}
\deg(\overline \Theta, x)=\deg(\overline \Theta, y)\end{equation*} 
for all $x,y\in T^2$ with $\overline \Theta(x)=
\overline \Theta(y)$. 
\end{enumerate} 
\end{lemma}
This last parabolicity criterion will be important for us. The
condition stipulates that
the local degree $\deg(\overline \Theta, \cdot)$ is constant on the
fiber $\overline \Theta^{-1}(p)$ 
for each $p\in S^2$. 

\subsection{Latt\`{e}s-Type Maps}
\label{sec:lattes-type-maps-1} A map $L\:\R^2 \ra \R^2$
is called {\em $\R$-linear} if
\begin{equation*}
  L(x+y)
  =
  L(x)+L(y)
  \quad\text{and}\quad
  L(\lambda x)=\lambda L(x)
\end{equation*}
for all $x,y\in \R^2$ and  $\lambda\in \R$. In other words, an $\R$-linear map $L\:\R^2 \ra \R^2$ is a linear map on $\R^2$ considered as a vector space over $\R$. We write $\det(L)\in \R$ for the determinant of $L$. 
 
A map  $A\: \R^2\ra \R^2$ is called {\em affine} if it can be represented  in the form
$$ A (u)=L_A(u)+a, \quad u\in \R^2,$$
where  $a\in \R^2$ and $L_A\: \R^2\ra \R^2$ is $\R$-linear.
The map $L_A$ is uniquely determined by $A$ and called  the {\em linear part} of $A$. 

We can now give a precise definition of a {\em Latt\`es-type map}.

\begin{definition}[Latt\`es-type maps]
  \label{def:Lattestype}
  Let $f\:S^2\ra S^2$ be a map,
  such that there exist a crystallographic group $G$, an affine
  map $A\: \R^2\ra \R^2$ with $\det(L_A)>1$ that is
  $G$-equivariant, and a branched covering  map $\Theta\: \R^2\ra
  S^2$ induced by $G$ such that  $f\circ \Theta=\Theta\circ
  A$. Then $f$ is called a  {\em Latt\`es-type map}.
\end{definition} 

Note that then we have a commutative diagram as in
\eqref{eq:def_lattes_type}.  It follows from condition~\ref{item:Lattessruciii} in Theorem~\ref{thm:Lattesstruc} that every Latt\`es map is also of Latt\`es-type. 

Latt\`es-type maps are natural non-holomorphic analogs of
Latt\`es maps.
In this context, we usually write
$\R^2$
instead of $\C$ for the plane, to emphasize 
that we do not rely on a  complex
structure.

If $f\:S^2 \ra S^2$ is a Latt\`es-type map,  
$G$ is a crystallographic group, and $\Theta\: \R^2\ra S^2$ is induced by $G$ as in Definition~\ref{def:Lattestype}, then $G$  is necessarily
of {\em non-torus type}, meaning that $G$ is not isomorphic to the rank-2 lattice $\Z^2$. This implies that there is a natural identification  
$\R^2/G\cong S^2$ of the quotient space $\R^2/G$ with the
underlying $2$-sphere $S^2$.
Under this identification, $\Theta$
corresponds to the quotient map $\R^2\to \R^2/G$ (see \cite[Section 3.4]{BM17}).  


In the following, we summarize some facts about these
maps.
Note first, that if $f$
is a Latt\`es-type map, and $A$ is
as in Definition~\ref{def:Lattestype} with its linear part $L_A$,
then $\det(L_A)=\deg(f)\ge 2$ (see \cite[Lemma 3.16]{BM17}).
This is the underlying reason
for the requirement $\det(L_A)>1$ in
Definition~\ref{def:Lattestype}.

Some of the relations between Latt\`{e}s-type maps, quotients of torus
endomorphisms, and Thurston maps with parabolic orbifold are  covered by the
following two results. 

\begin{prop}
  \label{prop:immediate}
  Every Latt\`es-type map $f\: S^2\ra S^2$  is a quotient of a
  torus endomorphism and hence a Thurston map. It has a parabolic
  orbifold and no periodic critical points.
\end{prop}
This is \cite[Proposition~3.5]{BM17}. 

\begin{prop}
  \label{prop:fpara_Latt_type}
  A Thurston map $f\: S^2\ra S^2$ is Thurston equivalent to a
  Latt\`{e}s-type map  if and only if it has
  parabolic orbifold and no periodic critical points. 
\end{prop}
This is \cite[Proposition~3.6]{BM17}.
Therefore, by \eqref{eq:list}, the signature of the
orbifold of a Latt\`{e}s-type map is in the list
\begin{equation}
  \label{eq:sign_Latt_type}
  (2,4,4), (2,3,6), (3,3,3), (2,2,2,2).
\end{equation}
 Only the last signature leads to maps that are 
genuinely different from Latt\`{e}s maps.
\begin{prop}
  \label{prop:Latt_Latt_type}
  A Latt\`{e}s-type map $f\colon S^2\to S^2$ whose orbifold has
  signature $(2,4,4)$, $(2,3,6)$, or $(3,3,3)$ is topologically
  conjugate to a Latt\`{e}s map. 
\end{prop}
This is  \cite[Proposition~3.18]{BM17}. A Latt\`{e}s-type map
with orbifold signature $(2,2,2,2)$ is in general not
topologically conjugate (or Thurston equivalent) to a Latt\`{e}s
map (see \cite[Theorem~3.22]{BM17} and
\cite[Proposition~9.7]{DH}).   These Latt\`{e}s-type maps  are examples of 
{\em nearly Euclidean Thurston maps};  see \cite{CFPP} for the
definition.

A Latt\`{e}s-type map is not necessarily expanding as a Thurston map
(see \cite[Example~6.15]{BM17}).
To record a criterion
for this, we call  an $\R$-linear map
 $L\colon \R^2\to \R^2$
 {\em expanding} if $|\lambda|>1$ for each of the
(possibly complex) roots $\lambda$ of the characteristic polynomial 
$P_L(z)\coloneqq  \det(L-z\id_{\R^2})$   of $L$. 

\begin{prop}
  \label{prop:Lattes_type_exp}
  Let $f\colon S^2\to S^2$ be a Latt\`{e}s-type map and $L$ be
  the linear part of an affine  map $A$ as in 
  Definition~\ref{def:Lattestype}. Then $f$ is expanding (as a Thurston map) if and
  only if $L$ is expanding (as a linear map). 
\end{prop}

This is \cite[Proposition~6.13]{BM17}. 

\subsection{Lattices and Tori}
\label{sec:tori}

We quickly review some facts about lattices and tori
(see   \cite[Section A.8]{BM17} for more details).
   
A {\em lattice}\index{lattice}  
$\Gamma \sub \R^2$ is a non-trivial discrete subgroup of $\R^2$ (considered as a group with vector addition). 
The {\em rank}
 of a lattice is the dimension of the subspace of $\R^2$ 
(considered as a  real vector space) spanned by the elements in
$\Gamma$. Here we are only interested in {\em rank-2} lattices
$\Gamma$, i.e.,  lattices $\Gamma\sub \R^2$ that span $\R^2$. 

 If $\Gamma \sub \R^2$ is  a rank-$2$ lattice, then  
 the quotient space $\R^2/\Gamma$ (equipped with the quotient topology) is a $2$-dimensional 
torus $T^2$,
and the quotient map $\pi\: \R^2 \ra 
T^2=\R^2/\Gamma$ is a covering map. 
The lattice translations $\tau_\ga$, $\ga \in \Gamma$, are deck transformations of 
the quotient map $\pi$ and so $\pi=\pi\circ \tau_{\ga}$ for $\ga\in \Gamma $.
 Actually,  every deck transformation of $\pi$ has this form (see  \cite[Lemma A.25~(i)]{BM17}).  
 Conversely, we may assume that any topological torus $T^2$ is   given in the form $T^2=\R^2/\Gamma$ with some rank-2 lattice $\Gamma \sub \R^2$. 

 In the following lemma, we collect
 various statements that are used
later.

\begin{lemma}\label{lem:torilifts}
Let $\Gamma\sub \R^2$ be a rank-2 lattice, $T^2=\R^2/\Gamma$, and 
$\pi\: \R^2\ra T^2=\R^2/\Gamma$ be the quotient map. 

 \begin{enumerate}
    
    \item 
    \label{item:tori2}
    If $\overline A\: T^2\ra T^2$ is a torus endomorphism, then
    $\overline A$ can be lifted to a homeomorphism on $\R^2$,
    i.e., there exists a homeomorphism $A\: \R^2\ra \R^2$, such
    that $\overline A\circ \pi =\pi \circ A$. The homeomorphism
    $A$ is orientation-preserving, and unique up to
    postcomposition with a translation $\tau_\ga$,
    $\ga\in \Gamma$.
    
      \item 
        \label{item:tori3}
        If $\overline A\: T^2\ra T^2$ is a torus endomorphism,
        then there exists a unique invertible $\R$-linear map
        $L\: \R^2\ra \R^2$ with $L(\Gamma)\sub \Gamma$, such that
        for every lift $ A$ as in \ref{item:tori2} we have
    \begin{equation*} \label{eq:mapL} 
A\circ \tau_\ga\circ A^{-1}= \tau_{L(\ga)}= L\circ \tau_\ga \circ L^{-1}
\end{equation*}
for all $\ga\in \Gamma $.

  \item 
    \label{item:tori4}
    If $\overline  A\: T^2\ra T^2$ is a torus endomorphism and
    $L$ is 
    the map as in  \ref{item:tori3}, then 
    $\deg(\overline A)=\det(L)$. 
 \end{enumerate}
\end{lemma}  
This is part of   \cite[Lemma A.25]{BM17}. Note that there it was not explicitly stated that the linear map $L$ in  \ref{item:tori3} is invertible. This  was addressed in the proof though: 
one observes  that  the inclusion $L(\Gamma)\sub \Gamma$ and the relation $A\circ \tau_\ga\circ A^{-1}= \tau_{L(\ga)}$ for $\ga \in \Gamma$ imply that the map $\ga\in \Gamma \mapsto L(\ga)\in \Gamma$ is injective.
Therefore, $L\: \Gamma \ra \Gamma$ is an injective group homomorphism.
Since $\Gamma$ is a rank-2 lattice,  $L$ must be invertible as an $\R$-linear map on $\R^2$.

One can identify $\Gamma$ with the fundamental group 
of $T^2$. Then the linear map $L$  is essentially 
the map on the fundamental group of $T^2$ induced by $\overline A$. 
For a careful explanation of this, see the discussion 
after \cite[Lemma A.25]{BM17}.

\subsection{Lifts by Branched Covering Maps}
\label{sec:lifts-branch-cover}

Since a Thurston map is a branched covering map, we need slight
variants of the standard lifting theorems for unbranched covering
maps. We list a useful uniqueness  result (this is essentially 
 \cite[Lemma~A.19~(i)]{BM17}).

\begin{lemma}
  \label{lem:liftbranched}
  Let $X,Y$, and $Z$ be surfaces and $f\colon X\to Y$ be a
  branched covering map. Suppose $g_1,g_2\colon Z\to X$ are
  continuous and discrete maps, such that
  $f\circ g_1 = f\circ g_2$. If there exists a point $z_0 \in Z$
  such that $p\coloneqq g_1(z_0)= g_2(z_0) $ and
  $f(p) \in Y{\setminus } f(\crit(f))$, then $g_1=g_2$.
\end{lemma}

Note that $g_1$ and $g_2$ can be considered as lifts of the map $h\coloneqq
f\circ g_1 = f\circ g_2$ under $f$.
Therefore, this is really
a uniqueness statement for lifts under $f$.

The condition $y\coloneqq f(p)\in  Y{\setminus } f(\crit(f))$ is the same  as  the requirement that   $y$ is not a critical value
of $f$, or equivalently, that the fiber $f^{-1}(y)$ contains no critical point of $f$.  We
will apply it in the case when $f\colon S^2\to S^2$ is a Thurston
map. Then this condition is satisfied if $p\in 
S^2{\setminus } f^{-1}(\post(f))$, because this implies that 
$f(p) \in   S^2{\setminus } \post(f)\sub  S^2{\setminus } f(\crit(f))$.

\section{Parabolicity of the Orbifold}
\label{sec:para} 

In this section we will prove the implication \ref{item:quo1}
$\Rightarrow$ \ref{item:quo2} in Theorem~\ref{thm:quo}.
Throughout the section, we assume that $f\: S^2\ra S^2$ is a
given quotient of a torus endomorphism that is expanding as a
Thurston map (see Lemma~\ref{lem:simobserv}~\ref{item:simobi}). Then there
exists a torus $T^2$, and maps $\overline A\: T^2\ra T^2$ and
$\overline \Theta\: T^2 \ra S^2$ as in \eqref{eq:Lattesdef}. We
can identify $T^2$ with a quotient $\R^2/\Gamma $, where
$\Gamma \sub \R^2$ is a rank-$2$ lattice, and we obtain a
quotient map $\pi \: \R^2\ra T^2\cong \R^2/\Gamma $.  The map
$\overline A$ lifts to an orientation-preserving homeomorphism
$A\: \R^2 \ra \R^2$, such that $\pi \circ A= \overline A\circ \pi$
(this is standard and explicitly formulated in
Lemma~\ref{lem:torilifts}~\ref{item:tori2}).  We define
$\Theta=\overline \Theta \circ \pi$. This is a branched covering
map, since $\pi$ is a covering map and $\overline \Theta$ is a
branched covering map (see
Lemma~\ref{lem:2_3_branched}~\ref{item:2_out3_1}). This leads to
the following commutative diagram:

\begin{equation}
  \label{eq:Athetafpi34}
  \xymatrix{
    \R^2 \ar[r]^{A} \ar[d]_{\pi} \ar@/_2pc/[dd]_{\Theta} 
    & \R^2 \ar[d]^{\pi}\ar@/^2pc/[dd]^{\Theta}
    \\
    T^{\smash{2}} \ar[r]^{\overline{A}}\ar[d]_{\overline{\Theta}}  & T^{\smash{2}}\ar[d]^{\overline{\Theta}}
    \\
    S^2 \ar[r]^{f} & S^2\rlap{.}
  }
\end{equation}



We denote by $G$ the group of all deck transformations of
$\Theta$, i.e., the group of all homeomorphisms $g: \R^2\ra
\R^2$, 
such that $\Theta\circ g=\Theta$. Since $\Theta$ preserves
orientation, the same is true for each homeomorphism $g\in G$.

Recall that $\tau_h$ for $h\in \R^2$ denotes the 
translation on $\R^2$ given by $\tau_h(u)=u+h$ for $u\in
\R^2$.
Then $\pi \circ \tau_\ga=\pi$ for $\ga\in \Gamma$.
This
implies that all lattice translations $\tau_\ga$,
$\ga\in \Gamma$, belong to $G$; indeed, for $\ga\in \Gamma$ we have 
$$ \Theta \circ \tau_\ga=\overline \Theta\circ \pi \circ  \tau_\ga
=\overline \Theta\circ \pi =  \Theta,$$
as required. 

Our goal is to show that $f$ has a parabolic orbifold. To do so,
we want to apply Lemma~\ref{lem:simobserv}~\ref{item:simobiii}. Essentially, we have to show that the group $G$ of deck
transformations acts transitively on each fiber of $\Theta$,
i.e., on each of the sets $\Theta^{-1}(p)$, $p\in S^2$.
This means that we have to analyze
some properties of the fixed map $\Theta$ in  
\eqref{eq:Athetafpi34}. Note that the diagram \eqref{eq:Athetafpi34}  remains valid with the same map $\Theta$, if we replace $A$ with 
$A'=\tau_\gamma\circ A$ for any $\gamma\in \Gamma$. We can also replace $f, \overline A, A$ with iterates $f^n,  \overline A^n, A^n$, respectively.  We will make such replacements whenever this is convenient. 

The map $A$ induces an invertible  linear map $L\: \R^2 \ra \R^2$ such that
$L(\Gamma)\sub \Gamma$ and
\begin{equation}\label{eq:ALtauga}
A \circ \tau_\ga \circ A^{-1}=\tau_{L(\ga)}= L \circ \tau_\ga \circ L^{-1}
\end{equation}  
for all $\ga \in \Gamma$ (see Lemma~\ref{lem:torilifts}~\ref{item:tori3}). As we mentioned, this  map  $L$ can be viewed as the homomorphism  induced by $\overline A$ on the fundamental group on $T^2$ (see the discussion after Lemma~\ref{lem:torilifts}). If we use \eqref{eq:ALtauga} repeatedly, then we see that 
\begin{equation}\label{eq:ALtauga2}
A^n \circ \tau_\ga \circ A^{-n}=\tau_{L^n(\ga)}= L^n \circ \tau_\ga \circ L^{-n}
\end{equation}  
for all $n\in \N$ and $\ga \in \Gamma$.

 If  $x\in \R^2$ and $\ga\in \Gamma$, then \eqref{eq:ALtauga} implies that 
\begin{align*}
|A(x+\ga)-L(x+\ga)|&= |(A\circ \tau_\ga)(x)-L(x)-L(\ga)|\\ &
=|(\tau_{L(\ga)}\circ A)(x)-L(x)-L(\ga)|\\
&= |A(x)-L(x)|.
\end{align*} 
Since the lattice translations $\tau_\ga$, $\ga\in \Gamma$, act cocompactly on $\R^2$,  it follows that there exists a constant $C_0\ge 0$, such that 
\begin{equation}
  \label{eq:LAclose}
  \abs{A(x)-L(x)}\le C_0
  \quad \text{for $x\in \R^2$.} 
\end{equation}
Therefore, the maps $A$ and $L$  agree
``coarsely'' on large scales.

Before we go into more details, we outline the ensuing argument. 
Since $f$ is an expanding  Thurston map, we first want to
translate this expansion  property of $f$  into expansion
properties for the above maps $A$ and $L$. In particular, $L$ is
an expanding linear map (Corollary~\ref{cor:Lexp}).
As we already mentioned, to prove that $f$ has
a parabolic orbifold (see Proposition~\ref{prop:quottorendo=para}),
 we have to show that    $G$ acts transitively on the fibers 
$\Theta^{-1}(p)$, $p\in S^2$ (see
Lemma~\ref{lem:decktransonfib}).   In \cite{Mi06}  one can find related considerations for Latt\`es maps. There the holomorphicity of the underlying maps is crucially used. Here we will instead give a dynamical argument relying  on the  expansion property of $A$. 
We now proceed to establishing  the details. 

\subsection{Expansion Properties} 
\label{sec:expansionprop} 
We start with expansion properties of $A$. Actually, it is easier
to formulate and prove  contraction properties  of $A^{-1}$.
In the following, all metric notions
on $\R^2$ refer to the Euclidean metric and all metric notions on $S^2$ to a fixed visual metric $\varrho$ for $f$ with expansion factor $\Lambda>1$.

\begin{lemma} 
  \label{lem:unifcontTheta} 
  Let $\Theta\: \R^2 \ra S^2$ be a map as \eqref{eq:Athetafpi34}. 
  Then the following statements are true: 
  \begin{enumerate}
    
  \item
    \label{item:unifcontTheta1}
    For each $\eps>0$, there exists $\delta>0$, such that 
    for all $x,y\in \R^2$, we have
    \begin{equation*}
      |x-y|< \delta
      \Rightarrow
      \varrho (\Theta(x), (\Theta(y))<\eps. 
    \end{equation*}
    
  \item 
    \label{item:unifcontTheta2}    
    For each $\eps>0$, there exists $\delta>0$, such that 
    for each connected set  $K\sub \R^2$, we have
    \begin{equation*}
      \diam_\varrho (\Theta(K))
      <
      \delta \Rightarrow  \diam(K) <\eps.
    \end{equation*}
  \end{enumerate}
\end{lemma}

The statement and its proof are a small modification of the
similar statement \cite[Lemma 6.14]{BM17}.

\begin{proof}  \ref{item:unifcontTheta1} The assertion is that 
$\Theta$ is uniformly continuous on $\R^2$. Essentially, this follows from the fact that  the group $G$ of deck transformations of $\Theta$ 
contains the subgroup $G'\coloneqq\{\tau_\ga: \ga \in \Gamma\}$
of   all lattice translations and that   this subgroup $G'$ acts isometrically and  cocompactly on $\R^2$. 

In particular,   we can find a compact fundamental domain $F\sub \R^2$ for the action 
of $G'$ on $\R^2$. Now suppose $x,y\in \R^2$ and $ |x-y|$ is small.
 Then there exists 
$g\in G'$, such that $g(x)\in F$. If $ |x-y|$ is small enough, as we may assume,  then $g(x), g(y)\in U$, where $U$ is a fixed compact neighborhood of $F$. Since $\Theta$ is uniformly continuous on $U$, and $|g(x)-g(y)|=|x-y|$, it follows that 
$$  \varrho(\Theta(x), \Theta(y) )= \varrho(\Theta(g(x)), \Theta(g(y) )$$ 
is small only depending on $ |x-y|$. The uniform continuity of $\Theta$ follows. 

\smallskip
\ref{item:unifcontTheta2}
We argue by contradiction and assume that the statement is
false. Then there exist connected sets $K_n\sub \R^2$, such that 
$\diam_\varrho(\Theta(K_n))\to 0$ as $n\to \infty$, but
$\diam(K_n)\ge \eps_0$ for $n\in \N$, where $\eps_0>0$.

  We pick a point $x_n\in K_n$ for $n\in \N$. If we replace each
  set $K_n$ with its image  $K'_n=g_n(K_n)$ for suitable
  $g_n\in G' $, where again $G'\coloneqq\{\tau_\ga: \ga \in \Gamma\}$   and pass to a subsequence if
  necessary, then we may assume that the sequence $\{x_n\}$
  converges, say $x_n\to x\in \R^2$ as $n\to \infty$. Note that $\diam(K'_n)=\diam(K_n)$ and
  $\Theta(K'_n)=\Theta(K_n)$. 
 
 Let $p\coloneqq \Theta(x)$.  Since $\Theta\: \R^2\ra S^2$ is a
 branched covering map, the set $\Theta^{-1}(p)$ is discrete in
 $\R^2$ and consists of isolated points. In particular, $x\in
 \Theta^{-1}(p)$ is an
 isolated point of $\Theta^{-1}(p)$ and so there exists a constant
 $m>0$ such that $|y-x|\ge m$ whenever $x,y\in \Theta^{-1}(p)$ and $x\ne y$.  
 
 Pick a constant $c$ with $0<c<\min\{\eps_0/2, m\}$.  The set $K_n$ is connected, and has diameter $\diam(K_n)\ge \eps_0> 2c$. Hence $K_n$ cannot be contained in the disk $\{z\in \R^2: |z-x_n|<c\}$, and so  it meets the circle $\{z\in \R^2: |z-x_n|=c\}$. It follows  that  there exists a
 point $y_n \in K_n$ with $|x_n-y_n|=c$. By passing to another subsequence if necessary, we may assume that the sequence $\{y_n\}$ converges, say $y_n\to y\in \R^2$ as $n\to \infty$. 
 Then $|x-y|=c<m$.
 Note that $\Theta(x_n), \Theta(y_n)\in \Theta(K_n)$ for $n\in \N$, 
and $\diam_\varrho(\Theta(K_n))\to 0$ as $n\to \infty$. 
Therefore
\begin{equation*}
  p
  =
  \Theta(x)
  =
  \lim_{n\to \infty} \Theta(x_n)
  =
  \lim_{n\to \infty} \Theta(y_n) =\Theta(y), 
\end{equation*}
and  $x,y\in \Theta^{-1}(p)$. Since $|x-y|=c>0$, we have $x\ne y$. Then  $x$ and $y$ are two distinct 
points in $\Theta^{-1}(p)$ with $|x-y|=c<m$. This contradicts the choice of $m$, and the statement follows. 
 \end{proof}

After this preparation, we now turn to the  contraction properties of the map $A^{-1}$. 

\begin{lemma}
  \label{lem:exptranstoA}
  Let  the map $A\: \R^2 \ra \R^2$ be  as in
  \eqref{eq:Athetafpi34}. If $\eps_1, \eps_2>0$, then there
  exists $n_0\in \N$, such that
  \begin{equation*}
    \abs{A^{-n}(x)-A^{-n}(y)}
    \le
    \eps_1\abs{x-y}+\eps_2
  \end{equation*} 
  for all $x,y\in \R^2$ and $n\in \N$ with $n\ge n_0$. 
\end{lemma}
The lemma essentially says that high iterates of $A^{-1}$ shrink
distances that are not too small by an arbitrarily small
factor. Conversely, by applying the statement to $x=A^n(u)$ and
$y=A^n(v)$ for $u,v\in \R^2$, we see that sufficiently high
iterates of $A$ expand distances that are not too small by an
arbitrarily large factor.

\begin{proof}
  Let $\varrho$ be the visual metric for $f$ on $S^2$ that we
  fixed earlier and $\Lambda>1$ be the corresponding expansion
  factor. We also fix constants $\delta_\varrho>0$
  and $C>0$ as in  Lemma~\ref{lem:liftpathshrinks}. Then, by Lemma~\ref{lem:unifcontTheta}~\ref{item:unifcontTheta1}, we
  can find $\delta>0$ with the following property: if $\beta $ is
  a path in $\R^2$ with $\diam(\beta) <\delta$ and
  $\alpha=\Theta\circ \beta$, then
  $\diam_{\varrho}(\alpha) <\delta_{\varrho}$.

  Now suppose that $\beta$ is a path in $\R^2$ with
  $\diam(\beta)<\delta$. Then the corresponding path
  $\alpha=\Theta\circ \beta$ satisfies $\diam_{\varrho}(\alpha)
  <\delta_{\varrho}$. For $n\in \N$, we also consider  the  paths
  $\widetilde{\beta}_n\coloneqq A^{-n} 
  \circ \beta$ in $\R^2$, 
  and $\widetilde{\alpha}_n \coloneqq \Theta\circ
  \widetilde{\beta}_n = \Theta\circ A^{-n}
  \circ \beta$ in $S^2$.
  From \eqref{eq:Athetafpi34} we obtain
  \begin{equation*}
    f^n \circ \widetilde \alpha_n
    =
    f^n \circ \Theta\circ A^{-n}\circ \beta
    =
    \Theta\circ A^n \circ A^{-n}\circ \beta
    =
    \Theta \circ \beta=\alpha.  
  \end{equation*}
  Therefore, $\widetilde \alpha_n$ is a lift
  of $\alpha$ under $f^n$. The
  relations between the paths are summarized in the following
  commutative diagram:
  \begin{equation} \label{eq: commpaths}
   \xymatrix{
     \widetilde{\beta}_n \;\,(\text{in } \R^{2}) \ar[r]^{A^n}
     \ar[d]_{\Theta} & \beta \;\,(\text{in } \R^{2}) \ar[d]^{\Theta}
     \\
     \widetilde{\alpha}_n \;\,(\text{in } S^{2}) \ar[r]^{f^n} & \alpha
     \;\,(\text{in } S^{2}) \rlap{.} 
   }
 \end{equation}

 Since $\diam_\varrho(\alpha) < \delta_\varrho$,  
 by the definition of $C$ and $\delta_\varrho$ we have 
  \begin{equation}
    \label{eq:liftpathfn}
    \diam_\varrho(\widetilde{\alpha}_n)
    \le
    C\Lambda^{-n}. 
 \end{equation}

 Let $\eps_1,\eps_2>0$ be arbitrary, and
 $\widetilde{\delta}\coloneqq \min\{\eps_1\delta, \eps_2\}>0$,
 where $\delta>0$ is chosen as above. Then by
 Lemma~\ref{lem:unifcontTheta}, we can choose
 $\widetilde{\delta}_{\varrho}>0$, such that
 $\diam(\widetilde{\beta}_n)<\widetilde{\delta}$ whenever
 $\diam_\varrho(\widetilde{\alpha}_n)<\widetilde{\delta}_{\varrho}$
 in \eqref{eq: commpaths}.  Note that the constants
 $\delta, \delta_\varrho, \widetilde{\delta}_\varrho,
 \widetilde{\delta}$ serve to control the diameters of
 $\beta,\alpha, \widetilde{\alpha}_n, \widetilde{\beta}_n$,
 respectively.

 Choose $n_0\in \N$, such that
 $C\Lambda^{-n}<\widetilde{\delta}_\varrho$ for $n\ge n_0$. 
 Then  
 \begin{equation}
   \label{eq:pathcontracAn} 
   \diam(\widetilde{\beta}_n)
   <
   \widetilde{\delta}
   =
   \min\{ \eps_1\delta, \eps_2\} 
\end{equation}  
for all $n\in \N$ with $n\ge n_0$.

Now suppose  $x,y\in \R^2$ are arbitrary, and let  $S$ be  the line segment joining $x$ and $y$. Then,
$S$ can be broken up into $N\in \N$ line segments of diameter $<\delta$, where $N\le |x-y|/\delta+1$. We can apply the previous considerations for each of these smaller (parametrized) line segments in the role of $\beta$. By what we have seen,  for $n\ge n_0$, each of these smaller line segments has an image under $A^{-n}$ of diameter $<\widetilde{\delta}$ by \eqref{eq:pathcontracAn}. 
 Since the concatenation of these $N$ image paths  is the path 
 $A^{-n}(S)$  connecting $A^{-n}(x)$  and $A^{-n}(y)$, we
 conclude that
\begin{align*}
 |A^{-n}(x)-A^{-n}(y)|& \le \diam(A^{-n}(S))\le N \widetilde{\delta}\\
 &\le  (|x-y|/\delta+1)\min\{ \eps_1\delta, \eps_2\}\le \eps_1|x-y|+\eps_2 
\end{align*} 
for $n\ge n_0$, as desired. 
\end{proof}

Recall  that an $\R$-linear map $L\: \R^2 \ra \R^2$ is called expanding
 if $|\lambda|>1$ for each of the (possibly complex) roots 
  $\lambda$ of the characteristic polynomial $P_L(z)=\det(L-z \id_{\R^2})$ of $L$.

  \begin{cor}
    \label{cor:Lexp} 
    Suppose the linear map $L\: \R^2 \ra \R^2$ is as in
    \eqref{eq:ALtauga}. Then $L$ is expanding.
  \end{cor} 

  \begin{proof}
    We argue by contradiction and assume that $L$ is not
    expanding.  Choosing $\eps_1=\eps_2=1/2$ in
    Lemma~\ref{lem:exptranstoA}, we can find a number $n\in \N$,
    such that
\begin{equation}\label{eq:expAk} 
|A^n(u)-A^n(v)|\ge 2|u-v|-1
\end{equation} 
for all $u,v\in \R^2$.
In other words, $A^n$ expands large distances roughly by the factor $2$.

Let $\lambda_1, \lambda_2\in \C$ be the two (possibly identical)
roots of the characteristic polynomial
$P(z)=\det(L-z \id_{\R^2})$ of $L$. We may assume that
$|\lambda_1|\le |\lambda_2|$. Since $P$ has real coefficients, we
have $\lambda_2=\overline{ \lambda_1}$ if $\lambda_1$ is not
real.  Moreover,
$\lambda_1\lambda_2=\det(L)=\deg(\overline A)=\deg(f)\ge 2$ (see
Lemma~\ref{lem:torilifts}~\ref{item:tori4} and
Lemma~\ref{lem:simobserv}~\ref{item:simobi}).
Therefore, the only possibility
that $L$ can fail to be expanding is if $\lambda_1$
is real and $|\lambda_1|\le 1$.
Then, there exists
$e\in \R^2$,
$e\ne 0$, such that $L(e)=\lambda_1 e$.

Let $\ell=\{t e: t \in \R\}$ be the line spanned by 
$e$. If $u,v\in \ell$  are arbitrary, then 
\begin{equation}\label{eq:Lcorsecontr}
 |L^{n}(u)-L^{n}(v)|=|\lambda_1|^{n}|u-v| \le |u-v|. \end{equation}
If  $\Gamma$ is the lattice chosen as in the beginning of this section,  then 
 \eqref{eq:ALtauga2} shows that 
\begin{align*}
 A^n(x+\ga)-L^n(x+\ga) &= (A^n\circ \tau_\ga)(x)-L^n(x)-L^n(\ga)\\
 &= (\tau_{L^n(\ga)}\circ A^n)(x)-L^n(x)-L^n(\ga) \\
 &=A^n(x)+ L^n(\ga) -L^n(x)-L^n(\ga) \\&
 = A^n(x) -L^n(x) 
 \end{align*} 
for all $x\in \R^2$. Since the lattice translations $\tau_\ga$, $\ga \in \Gamma$,  act cocompactly on 
$\R^2$,  this implies that  there exists a constant $C\ge 0$, such that 
$$ |A^n(u)-L^n(u)| \le C$$ for all $u\in \R^2$.   
Combining this with \eqref{eq:Lcorsecontr}, we see that 
$$ |A^n(u)-A^n(v)| \le  |L^n(u)-L^n(v)|  +2C \le |u-v|+2C$$ 
for all $u,v\in \ell$.
Therefore, the map $A^n$ expands 
distances along $\ell$ by at most an additive term.
This is irreconcilable with  \eqref{eq:expAk}, and we get a contradiction. The statement follows. 
\end{proof} 

We record the following consequence. 

\begin{cor}
  \label{cor:Anearfix} 
  Suppose that the map $A$ as in \eqref{eq:Athetafpi34} has a fixed
  point $x\in \R^2$.  Then if $U$ is any open neighborhood of
  $x$, we have $$\bigcup_{n\in \N_0}A^n(U)=\R^2. $$ Moreover, if
  $U$ is bounded in addition, then $U\sub A^n(U)$ for all
  sufficiently large $n\in \N$.
\end{cor} 

\begin{proof}
  Let $U$ be a neighborhood of $x$.
  Then, there exists $\eps>0$, such that
  $B\coloneqq \{ z\in \R^2: |z-x|<\eps\} \sub U$. If
  $y\in \R^2$ is arbitrary, then Lemma~\ref{lem:exptranstoA}
  implies that $|A^{-n}(y)-x|=|A^{-n}(y)-A^{-n}(x)|$ is
  arbitrarily small for $n\in \N$ sufficiently large. Hence there
  exist $n\in \N$, such that $A^{-n}(y)\in B\sub U$, and so
  $y\in A^{n}(U)$.  It follows that
  $\R^2=\bigcup_{n\in \N_0} A^n(U)$.

  If $U$ is bounded, then there exists $R>0$, such that
  $U\sub B'\coloneqq \{ z\in \R^2: |z-x|<R\} $. Applying
  Lemma~\ref{lem:exptranstoA} for $\eps_1=\eps/(2R)$ and
  $\eps_2=\eps/2$, we see that
$$A^{-n}(U)\sub A^{-n}(B')\sub B\sub U $$ for all sufficiently large $n\in \N$. 
Hence, $U\sub A^n(U)$ for all large $n$. 
\end{proof} 

\subsection{Transitive Action on Fibers}
\label{sec:transitive} 

Next, we will show that the group $G$ of deck transformations of the map $\Theta$ as in  \eqref{eq:Athetafpi34} acts transitively on each  fiber
$\Theta^{-1}(p)$, $p\in S^2$. We first show that this is
true in a  special case.

\begin{lemma}
  \label{lem:restrtrans}
  Let $p\in S^2{\setminus } f^{-1}(\post(f))$,
  $x,y\in \Theta^{-1}(p)$, and define $\bar x= \pi(x),
  \bar y=\pi(y)\in T^2$. Suppose $\bar x$ is a fixed point of
  $\overline A$. If $\bar y$ is also a fixed point of
  $\overline A$ or if $\bar x= \overline A(\bar y)$, then there
  exists $g\in G$, such that $g(x)=y$.
\end{lemma}

\begin{proof}
  Let the points $p\in S^2$, $x,y\in \R^2$, $\bar{x},\bar{y}\in T^2$ be given as in the
  statement. In particular, we assume that
  $\overline{A}(\bar{x})= \bar{x}$. Note that 
 $$ \pi(A(x))=(\pi\circ A)(x)= (\overline A\circ \pi)(x)=\overline A (\overline x)=\overline x, $$
  and so $x, A(x)\in  \pi^{-1}(\bar{x})$.   
This means that
  $\gamma_0\coloneqq x-A(x)\in \Gamma$. Define
  $A_0(u) = A(u) +\gamma_0=(\tau_{\gamma_0} \circ A)(u)$, $u\in
  \R^2$. Note that then $A_0(x)=A(x)+\gamma_0=x$, and so $ A_0$ has the fixed point $x$. Recall from  the discussion following
  \eqref{eq:Athetafpi34} that
  we may replace $A$ in this diagram with   $A_0$ (while all the other maps remain the same). 
 In other words, we are reduced to the case when $A(x)=x$ in addition to our other hypotheses.

  We now  consider the cases
  $\overline{A}(\bar{y})=\bar{y}$ and
  $\overline{A}(\bar{y})=\bar{x}$ separately.

  \smallskip
  {\em Case~I:} $\overline A(\bar y)=\bar x$.
  \smallskip

  This is the easy case. Note that 
$$
\pi(A(y))=(\pi\circ A)(y)= (\overline A\circ \pi)(y)=\overline A (\overline y)=\overline x,$$
 and so $A(x), A(y)\in
  \pi^{-1}(\bar{x})$. This implies that we can find  $\gamma\in \Gamma$ with
  $A(y)-A(x)=\gamma$. Then
  \begin{equation*}
    y=A^{-1}(A(x) + \gamma). 
  \end{equation*}
  Thus, $g\coloneqq A^{-1} \circ \tau_\ga \circ A$ is a
  homeomorphism on $\R^2$ with $g(x)=y$.

  We want  to show that $g\in G$, meaning we need to verify
  that $\Theta\circ g = \Theta$. 
  We know that $\tau_\ga \in G$. Using $f\circ \Theta=
  \Theta\circ A$ from \eqref{eq:Athetafpi34},  we obtain
  \begin{align*}
    f\circ \Theta\circ g
    &=
      \Theta \circ A \circ g= \Theta \circ \tau_\ga\circ A \\
    & =
      \Theta \circ A =f\circ \Theta.
  \end{align*}
 We now  apply  Lemma~\ref{lem:liftbranched} for the branched covering maps 
 $\Theta$ and $\Theta\circ g$. Note that
  $(\Theta\circ g)(x)=\Theta(y)=\Theta(x)=p$ and 
 $p\in S^2{\setminus } f^{-1}(\post(f))$. It follows that
  $\Theta=\Theta\circ g$. We proved the statement in Case I.

\smallskip
  {\em Case~II:}
  $\bar x$ and $\bar y$ are fixed points of $\overline A$.
\smallskip

This case is much harder, since there is no translation $\tau_\gamma$
with $\gamma\in \Gamma$ that maps $A(x)$ to $A(y)$. To construct
a deck transformation of $\Theta$ as in the statement, we first
show that we can obtain a local one.
  
  \smallskip
  \emph{Claim 1.} There is a homeomorphism $\tilde{g}\colon
  U\to V$ between  bounded and connected open neighborhoods $U$ and
  $V$ of
  $x$ and $y$, respectively, with $\Theta\circ \tilde{g} =
  \Theta$ on $U$.
  \smallskip
  
  To prove this, we note that our assumption
  $p\in S^2{\setminus } f^{-1}(\post(f))\sub S^2{\setminus }\post(f) $
  implies that $\overline \Theta$, and hence also $\Theta$, has
  no critical point over $p$, because 
  $\post(f) = \overline{\Theta}(\crit(\overline{\Theta}))$ (see 
   Lemma~\ref{lem:simobserv}~\ref{item:simobii}). In particular, $\Theta$ is a local homeomorphism
  near both points $x,y\in \Theta^{-1}(p)$. This implies  that there
  exist bounded and connected   open neighborhoods $U\subset \R^2$
  of $x$, $V\subset \R^2$ of $y$, and  $W\subset S^2$ of
  $p$, such that $\Theta|U\colon U\to W$ and
  $\Theta|V\colon V\to W$ are homeomorphisms. Defining
  $\tilde{g}\coloneqq (\Theta|V)^{-1} \circ (\Theta|U)$ on $U$
  gives the desired map, proving Claim~1.  \smallskip
  
  Now, the idea is to extend $\tilde g$ to a deck transformation
  on $\R^2$ using the dynamics of $A$ near its fixed point
  $x$. By Corollary~\ref{cor:Anearfix}, we know that
  $U\sub A^n(U)$ for all sufficiently large $n\in \N$. Replacing
  $A$ with such an iterate $A^n$ (and consequently $f$ with
  $f^n$, and $\overline{A}$ with $\overline{A}^n$), we may assume
  that $U \subset A(U)$.
  Note, we then still have
  $\overline A(\overline x)=\overline x$ and
  $\overline A(\overline y)=\overline y$ for the new map
  $\overline A$. We make the assumption $U\sub A(U)$ from now on.
  
We know that  $\overline A(\overline y)=\overline y$, but, in general, the point $y$ will not be a fixed point of $A$. 
 We have 
 $$ \pi(A(y))=(\pi\circ A)(y)= (\overline A\circ \pi)(y)=\overline A (\overline y)=\overline y,$$
 and so $y,A(y)\in \pi^{-1}(\overline y)$.
 It follows that   $A(y)=y+\ga$ for some
  $\ga\in \Gamma$. Therefore, if we define
  $\tilde A(u) = A(u)- \gamma= (\tau_{-\gamma} \circ A)(u)$ for
  $u\in \R^2$, then $\tilde A(y)=y$. 

  Define $U^n= A^n(U)$ for $n\in\N_0$. The sets $U^n$ are
  connected open sets containing $x$. We have $U^{n}\sub U^{n+1}$
  for $n\in \N_0$, and $\bigcup_{n\in \N_0}U^n =\R^2$.  The last
  fact follows from Corollary~\ref{cor:Anearfix}.

  We now define homeomorphisms $g_n$
  mapping $U^n$ into $\R^2$ recursively, by setting
  $g_0\coloneqq\tilde g$ on $U^0=U$, and
  \begin{equation*}
    g_{n+1}
    \coloneqq
    \tilde A  \circ g_n \circ A^{-1}
    \quad \text{on } U^{n+1}
  \end{equation*}
  for $n\in \N_0$. Note that this makes sense, because
  $A^{-1}(U^{n+1})=U^n$. This definition implies that
  \begin{equation*}
    \tau_\ga\circ g_{n+1}\circ A
    =
    A\circ g_n \quad \text{on } U^n
  \end{equation*}
  for all $n\in \N_0$.   

  One verifies by induction that $\Theta\circ g_n = \Theta$ on
  $U^n$ for all $n\in \N$.  Indeed, this is true for $n=0$ by
  definition of $g_0=\tilde g$.  If it is true for $n\in \N_0$,
  then it is also true for $n+1$, because 
\begin{align*}
  \Theta \circ g_{n+1}
  &= \Theta \circ \tau_{-\ga}\circ A \circ g_n \circ A^{-1} \\
  &=   \Theta  \circ A \circ g_n \circ A^{-1} \\
  &= f \circ \Theta  \circ g_n \circ A^{-1} =  f \circ \Theta  \circ A^{-1} \\
  &=  \Theta  \circ A  \circ A^{-1} =\Theta. 
\end{align*}   
By induction, one also shows that $g_n(x)=y$ for all $n\in
\N_0$. Indeed, $g_0(x)=\tilde g(x)=y$, and if this is true for
$n\in \N_0$, then it is also true for $n+1$, because
\begin{align*}
  g_{n+1}(x)
  &=
    (\tilde  A \circ g_n \circ A^{-1})(x)\\
  & =
    (\tilde A \circ g_n)(x) = \tilde  A(y) =y. 
\end{align*} 

\smallskip
\emph{Claim~2.} We have $g_{n+1}|U^n = g_n$ for all $n\in \N_0$.
\smallskip

To see this,  we want to apply Lemma~\ref{lem:liftbranched}
to the branched covering map
$\Theta\: \R^2\ra S^2$,
 and the maps $g_n$ and $g_{n+1}|U_n$.  
We know that $$\Theta\circ g_n=\Theta = \Theta\circ (g_{n+1}|U_n)$$ on the connected open set  $U_n$. Moreover, 
 $g_n(x)=y=g_{n+1}(x)$ and the point $y$
which lies in the fiber over $p=\Theta(y)$ not containing any
critical point of $\Theta$. Claim~2  follows.   \smallskip

Therefore, each homeomorphism
$g_n$ extends the previous one. Since the
sets $U^n$, $n\in \N_0$, exhaust $\R^2$, there exists a unique
map $g\: \R^2 \ra \R^2$, such that $g|U^n=g_n$ for all
$n\in \N_0$. It is clear that $g$ is continuous and  injective, because the maps $g_n$ have these
properties.  Moreover, it is clear that $g(x)=y$ and
$\Theta=\Theta\circ g$ on $\R^2$.

To finish the proof in the Case~II at hand, it remains to show
that $g\: \R^2\ra \R^2$ is surjective. To do this, let us shift our attention to
the images $g(U^n)$. 

\smallskip
\emph{Claim 3.} We have $g(U^n)= g_n(U^n)= \tilde{A}^n(V)$ for all $n\in
\N_0$. 
\smallskip

Recall that $V=\tilde g(U)$ was the neighborhood of $y$  defined in Claim~1. Thus, Claim~3 is true for
$n=0$, since $$g(U^0)=g_0(U^0)= \tilde g(U)=V =\tilde
A^0(V). $$ Moreover, if it is true for $n\in \N_0$, then it is also
true for $n+1$, because
\begin{align*}
  g(U^{n+1})
  &=
    g_{n+1}(U^{n+1})=g_{n+1}(A(U^n))\\
  &=
    (\tilde A \circ g_n \circ A^{-1})(A(U^n))
    =
    (\tilde A \circ g_n)(U^n)\\
  &=
    \tilde A(\tilde A^n(V)) = \tilde A^{n+1}(V),
\end{align*} 
proving Claim~3.
\smallskip

Recall that in \eqref{eq:Athetafpi34}, we can replace $A$ with 
$\tilde A= \tau_{-\ga}\circ A$. Since the map $\tilde A$ has the fixed
point $y$, we can  apply Corollary~\ref{cor:Anearfix} to $\tilde{A}$;
so the images of the neighborhood $V$ of $y$ under
iterates of $\tilde A$ will exhaust $\R^2$, i.e., we have
$ \bigcup_{n\in \N} \tilde A^n(V)=\R^2$. Using Claim~3, we
conclude that $g$ is surjective. This finishes the proof of
Case~II. The statement follows.  
\end{proof}

We now show transitivity of the action of $G$ on the fibers of
$\Theta$.

\begin{lemma}
  \label{lem:decktransonfib}
  Let $p\in S^2$ and $x,y\in \Theta^{-1}(p)$. Then there exists
  $g\in G$, such that $g(x)=y$.
\end{lemma} 

\begin{proof}
  We will first show the existence of a point
  $p_0\in S^2 {\setminus } \post(f)$ for which $G$ acts transitively
  on the fiber $\Theta^{-1}(p_0)$, and then deal with the general
  case.
 
  The periodic points of $f$ are dense in $S^2$ (see
  \cite[Corollary~9.2]{BM17}); in particular, we can find a
  periodic point $p_0\in S^2{\setminus } \post(f)$.  By replacing
  $f$ with suitable iterates $f^n$ (and $\overline{A}$ with $\overline{A}^{\,n}$),  we may assume that $p_0$ is a fixed point
  of $f$. Note that then we still have
  $p_0\in S^2{\setminus } \post(f)$, because the postcritical sets
  of a Thurston map and any of its iterates agree.
 
Since $p_0$ is a fixed point of $f$, the map  $\overline A$ sends  the set $\overline \Theta^{-1}(p_0)$
 into itself. It follows that each point in $\overline \Theta^{\,-1}(p_0)$ is either a periodic point of $\overline A$ or is mapped to a periodic point of $\overline A$  under all sufficiently high iterates of 
 $\overline A$. If we again replace $f$ and $\overline A$ with  carefully chosen  iterates, we can reduce ourselves to the following situation: $p_0\in S^2{\setminus } \post(f)$ is a fixed point of $f$, 
the set  $\overline \Theta^{-1}(p_0)$ contains at least one fixed point of $\overline A$, 
 and each point in  $\overline \Theta^{-1}(p_0)$ is either a fixed point of $\overline A$ or mapped to a fixed point by $\overline A$. Moreover, since $p_0$ is a fixed point of $f$, and $p_0\in S^2{\setminus } \post(f)$, we  have $p_0\in S^2{\setminus } f^{-1}(\post(f))$.  We pick a fixed point $\bar x\in \overline \Theta^{-1}(p_0)$ of $\overline A$, and a point $x\in \R^2$  with $\pi(x)=\bar x$. Then $x\in \Theta^{-1}(p_0)$.  
 
 Now, let  $y\in \Theta^{-1}(p_0)$ be arbitrary.  Then $\bar y\coloneqq\pi(y)\in \overline \Theta^{-1}(p_0)$. 
 By our choices, $\bar y$ is a fixed point of $\overline A$ or $\overline z\coloneqq\overline A(\overline y)$ is a fixed point of $\overline A$. 
 In the first case, there exists $g\in G$, such that $g(x)=y$ by the first part of Lemma~\ref{lem:restrtrans}. 
 
 In the second case, $\overline z$ is a fixed point of
 $\overline A$. Pick $z\in \R^2$, such that $\pi(z)=\overline
 z$. Then $z\in \overline \Theta^{-1}(p_0)$, and again there
 exists $g_1\in G$ with $g_1(x)=z$. Moreover, by
 Lemma~\ref{lem:restrtrans} there exists $g_2\in G$, such that
 $g_2(z)=y$. Then $g\coloneqq g_2\circ g_1\in G$ and $g(x)=y$.
 
 We see that $x$ can be sent to any  point in $\Theta^{-1}(p_0)$ by a suitable element $g$ in the group $G$; it follows that $G$ acts transitively on $\Theta^{-1}(p_0)$.  
 
 Now let $p\in S^2$ and $x,y\in \Theta^{-1}(p)$ be arbitrary. Then we can find a path 
 $\alpha\:[0,1]\ra  S^2$ with $\alpha(0)=p$ and $\alpha(1)=p_0$, so that $\alpha((0,1])\sub  S^2{\setminus } \post(f)$. Therefore, $\alpha$ lies in  $ S^2{\setminus } \post(f)$ with the possible exception of its  initial  point $p$.  Since $x,y\in \Theta^{-1}(p)$,  we can lift the path $\alpha$
  under the branched covering map $\Theta\:\R^2\ra S^2$ to paths $\beta_1,\beta_2\:[0,1]\ra \R^2$, such that 
 $$\Theta\circ \beta_1=\Theta\circ \beta_2 =\alpha, $$
  $\beta_1(0)=x$, and $\beta_2(0)=y$ (see \cite[Lemma~A.18]{BM17}).

 Let $x'=\beta_1(1)$ and $y'=\beta_2(1)$. Then 
 $$\Theta(x')=\Theta (\beta_1(1))=\alpha(1)=p_0, $$
 and so $x'\in  \Theta^{-1}(p_0)$. Similarly, $y'\in \Theta^{-1}(p_0)$. 
 
 By 
 the first part of the proof, there exists $g\in G$, such that $g(x')=y'$.
 Let $\beta_3\coloneqq g\circ \beta_1$. Then
 $$\beta_3(1)=g(\beta_1(1))=g(x')=y'=\beta_2(1), $$ and 
 $$\Theta\circ \beta_3=\Theta\circ g\circ \beta_1=\Theta\circ  \beta_1=\alpha. $$
Therefore, $\beta_3$ is a lift of $\alpha$ under $\Theta$ with the same
endpoint as $\beta_2$. Since $\Theta$ is a covering map over
$S^2{\setminus } \post(f)$ and $\alpha|(0,1]\sub S^2{\setminus } \post(f)$,
it follows that $\beta_2|(0,1]=\beta_3|(0,1]$
(see \cite[Lemma~A.6 (i)]{BM17}).
Hence $\beta_3(0)=\beta_2(0)$ by continuity,  and so 
$$ g(x)=g(\beta_1(0))=\beta_3(0)=\beta_2(0)=y. $$ 
This shows that $G$ acts transitively on the fiber $\Theta^{-1}(p)$.    
\end{proof} 

We are now ready to prove the implication \ref{item:quo1}
$\Rightarrow$ \ref{item:quo2} in Theorem~\ref{thm:quo}. 

\begin{prop} \label{prop:quottorendo=para}
Suppose $f\: S^2\ra S^2$ is an expanding Thurston map that is a quotient of a torus endomorphism. Then $f$ has a parabolic orbifold. 
\end{prop}

\begin{proof} We can use all the previous considerations 
for the maps as in \eqref{eq:Athetafpi34} and the deck transformation group $G$ of $\Theta$.

Let $p\in S^2$ be arbitrary,  and $\bar x, \bar y\in \overline \Theta^{-1}(p)$. 
Then, there exist points $x,y\in \R^2$, such that $\pi(x)=\bar x$ and $\pi(y)=\bar y$. 
Then $\Theta(x)=(\overline \Theta \circ \pi )(x) = \overline \Theta(\bar x)= p=\Theta(y). $
By Lemma~\ref{lem:decktransonfib} there exists $g\in G$, such that $g(x)=y$. Since $g$ is a homeomorphism, $\pi$ is a local homeomorphism, and both maps preserve orientation,  we have 
$ \deg(g, x)=1$ and  $\deg(\pi, x) =\deg(\pi, y)=1$. Since local degrees are multiplicative under compositions (see \eqref{eq:locdegmult}), we conclude  
\begin{align*}
\deg(\overline \Theta, \bar y)&= \deg(\overline \Theta, \bar y) \cdot \deg(\pi, y)= \deg(\Theta,y)\\
&= \deg(\Theta,y) \cdot \deg(g, x)= \deg (\Theta\circ g, x)\\
 &=\deg (\Theta, x) = \deg(\overline \Theta, \bar x).
\end {align*}  
Therefore, the local degree of $\overline \Theta$ in each fiber $\overline\Theta^{-1}(p)$, $p\in S^2$,  is constant. Now  Lemma~\ref{lem:simobserv}~\ref{item:simobiii} implies that $f$ has a parabolic orbifold.   
\end{proof}

\section{From Parabolic Orbifolds to Latt\`es-Type Maps}
\label{sec:charLtype}

In this section, we will prove the implication  \ref{item:quo2} 
$\Rightarrow$ \ref{item:quo3} in Theorem~\ref{thm:quo}. 
We first establish an auxiliary fact that helps us in identifying 
Latt\`es-type maps. 
\begin{lemma}
  \label{lem:Lattes_type_conj}
  Let $f\colon S^2\to S^2$ be  a map that is topologically conjugate to a  Latt\`{e}s-type map. Then $f$ itself is Latt\`{e}s-type map.\end{lemma}
  
As we will see, the proof is fairly straightforward. There is a
subtlety though that arises from the fact that, by definition, a
branched covering map is orientation-preserving (see
\eqref{eq:z^dbrcov}), but a homeomorphism as in the definition of
topological conjugacy may actually reverse orientation. To
address this, we have to compensate by complex conjugation on
$\C$ in a suitable way.

In the following, we denote by $\sigma\: \C\ra \C$, $\sigma(z)=\overline z$ for $z\in \C$, complex conjugation on  $\C$. Note that this is an $\R$-linear orientation-reversing isometry on $\C$ with  $\sigma^{-1}=\sigma$. 

\begin{proof}[Proof of Lemma~\ref{lem:Lattes_type_conj}]
 Suppose  $f\colon S^2 \to S^2$  is a map 
 that is topological conjugate to a Latt\`es-type map 
 $\widehat f\: \widehat S^2\ra  \widehat S^2$. Here   $\widehat S^2$ is another topologically $2$-sphere.   Both  $S^2$ and $\widehat S^2$ carry some fixed orientations. 

 According
  to Definition~\ref{def:Lattestype}, there
  exists a crystallographic group $ \widehat  G$ acting on $\C\cong \R^2$, an
  $ \widehat  G$-equivariant (real) affine map $ \widehat  A\colon \C\to \C$, and a
  branched covering map $ \widehat  \Theta\colon \C\to \widehat  S^2$ induced by
  $ \widehat  G$ with $ \widehat  f\circ  \widehat  \Theta =  \widehat  \Theta \circ  \widehat  A$. Moreover, 
  the linear part $L_{ \widehat A} $ of $ \widehat A$ satisfies 
  $\det(L_{ \widehat A}) >1$. 
  
 Since $f$ and $\widehat f $ are topologically conjugate, there exists a homeomorphism $h\:  \widehat S^2 \ra S^2$, such that 
 $f\circ h=h\circ  \widehat f$. We now have to distinguish two cases according to whether the homeomorphism $h\:  \widehat S^2 \ra S^2$ preserves or reverses orientation. We will treat the latter, slightly more difficult case in detail, and then comment on the small modifications for the former case. 
 
 Therefore, we now assume that $h\:  \widehat S^2 \ra S^2$ is orientation-reversing. We   define
 $\Theta\coloneqq h\circ  \widehat \Theta \circ \sigma$. 
 Then $\Theta\: \C\ra S^2$  is a branched covering map, as easily follows  from 
 the definitions and the fact that $ \widehat  \Theta\: \C\ra  \widehat  S^2$ is a branched covering map. Here, it is important that in the definition of $\Theta$, we compensate postcomposition of  $\widehat \Theta$ with  the orientation-reversing homeomorphism $h$ by precomposition with  the orientation-reversing homeomorphism  $\sigma$ to make $\Theta$ orientation-preserving. 

 We  conjugate everything else by $\sigma$. More precisely, we define 
 $G\coloneqq \{ \sigma \circ g \circ \sigma: g\in \widehat  G\} $.
 It is clear that $G$ is a crystallographic group on $\C$ and  
 $\Theta$ is induced by $G$.
 Moreover, we let
 $A \coloneqq  \sigma \circ \widehat A \circ \sigma$. Then, 
 $A$ is a (real) affine map on $\C\cong \R^2$ that is $G$-equivariant. For the linear part $L_A$ of $A$, we have 
 $L_A=\sigma\circ L_{ \widehat  A}\circ \sigma$, and so 
 $\det(L_A)=\det(L_{ \widehat  A})>1$. 

We can summarize the relations
  of all the  maps considered in the following 
  diagram, which is obviously commutative:
  \begin{equation*}
  \xymatrix{
    \C \ar[r]^{A} \ar[d]_{\sigma} \ar@/_2pc/[ddd]_{\Theta} 
    & \C \ar[d]^{\sigma}\ar@/^2pc/[ddd]^{\Theta}
    \\
    \C \ar[r]^{\widehat A}\ar[d]_{\widehat\Theta} 
    & \C\ar[d]^{\widehat \Theta}
    \\
   \widehat S^2 \ar[r]^{\widehat f}\ar[d]_h & \widehat S^2 \ar[d]^h
  \\
  S^2 \ar[r]^{f} & S^2 \rlap{.}
  }
\end{equation*}
It follows that $f$ is a Latt\`es-type map. 

If $h$ is orientation-preserving, the map $\sigma$ is not needed
in the previous argument. Formally, we can just replace $\sigma$
with the identity map on $\C$. Therefore, $f$ is also a
Latt\`es-type map in this case. The statement follows.
\end{proof}

After this preparation, we now prove the implication
\ref{item:quo2} $\Rightarrow$ \ref{item:quo3} in
Theorem~\ref{thm:quo}. 

\begin{prop}
  \label{prop:para_Lattes}
  Let $f\colon S^2\to S^2$ be an expanding Thurston map with
 a parabolic orbifold. Then $f$ is a Latt\`{e}s-type map. 
\end{prop}

\begin{proof}
  Let $f$ be as in the statement, and $\alpha=\alpha_f$ be the ramification function of $f$. 
 Since $f$ has a  parabolic 
 orbifold $\mathcal{O}_f=(\CDach, \alpha)$,  the signature of 
  $\mathcal{O}_f$ is in the list (see \eqref{eq:list})
  $$(\infty, \infty),\,  (2,2, \infty),\,  (2,4,4),\,  (2,3,6),\,  (3,3,3),\,  (2,2,2,2). $$  
  The fact that $f$, or equivalently $f^2$, is expanding (see
  \cite[Lemma~6.5]{BM17}), rules out the signatures
  $(\infty, \infty)$ and $(2,2, \infty)$. Indeed, if a Thurston
  map $f$ has one of these signatures, then $f$ or $f^2$ is a
  {\em Thurston polynomial} (i.e., there is a point that is
  completely invariant under the map); see the discussion after
  \cite[Lemma 7.4]{BM17}. However, no Thurston polynomial is expanding
  (see \cite[Lemma 6.8]{BM17}). This means that $\alpha$ does not
  attain the value $\infty$; so $f$ has no periodic critical points (see \cite[Proposition 2.2~(ii)]{BM17})  and the signature of  
   $\mathcal{O}_f$ is among 
  $$(2,4,4),\,  (2,3,6),\,  (3,3,3),\,  (2,2,2,2). $$  
  
  In the first three cases, $f$ has precisely three postcritical points, and is hence Thurston equivalent to a rational Thurston map $R\: \CDach \ra \CDach$  
  (see \cite[Theorem~7.2 and Lemma 2.5]{BM17}). Since $f$ and $R$ are Thurston equivalent,
  the orbifolds of $f$ and $R$ have the same signatures (see  
  \cite[Proposition 2.15]{BM17}), namely  $(2,4,4)$, $(2,3,6)$, or $(3,3,3)$.    
  In particular, $R$ has a parabolic orbifold.

  Moreover, the ramification function of $R$ only takes finite values which again implies that 
  $R$ has no periodic critical points. Hence, $R$ is expanding
   (see \cite[Proposition 2.3]{BM17}) and actually a Latt\`es map (see Theorem~\ref{thm:Lattesstruc}).  Since $f$ and $R$ are  Thurston equivalent, and both Thurston maps are expanding, it follows that $f$ and $R$ are topologically conjugate
   (see Theorem~\ref{thm:Th_equiv_top_conj}).
   Now Lemma~\ref{lem:Lattes_type_conj} implies that $f$ is a Latt\`es-type map (note that  the  Latt\`es map $R$ is also of Latt\`es-type). 
   
   It remains to consider the case where the signature of  
   $\mathcal{O}_f$ is equal to $(2,2,2,2)$. We know that $f$ is Thurston equivalent to a Latt\`es-type map $\widehat f\: \widehat S^2\ra \widehat S^2$ (see Proposition~\ref{prop:fpara_Latt_type}). The proof of  Proposition~\ref{prop:fpara_Latt_type} (see \cite[pp.~74--77]{BM17})
   shows that $f$ is a quotient of a torus endomorphism and that 
   the affine map $\widehat A$ for $\widehat f$ as in 
   Definition~\ref{def:Lattestype} can be chosen, so that its linear part $L_{\widehat A}$ is equal to the map $L$ as in 
   \eqref{eq:ALtauga}. It follows  that $L_{\widehat A}=L$  is expanding as a linear map (see Lemma~\ref{cor:Lexp}). This in turn implies that $\widehat f$ is expanding as a Thurston map
   (see Proposition~\ref{prop:Lattes_type_exp}).  Now we can again conclude that 
   $f$ and $\widehat f$ are topologically conjugate, and hence $f$ is a Latt\`es-type map by Lemma~\ref{lem:Lattes_type_conj}. 
\end{proof}

We can now wrap up the proof of Theorem~\ref{thm:quo}.

\begin{proof}[Proof of Theorem~\ref{thm:quo}]
The implications   \ref{item:quo1} 
$\Rightarrow$   \ref{item:quo2}  and \ref{item:quo2} 
$\Rightarrow$  \ref{item:quo3}  were proved in
Propositions~\ref{prop:quottorendo=para} 
and \ref{prop:para_Lattes}, respectively. The implication
\ref{item:quo3} $\Rightarrow$   \ref{item:quo1} was explicitly
stated in Proposition~\ref{prop:immediate}. 
\end{proof}

\end{document}